\documentclass[a4paper,reqno]{amsart}
\usepackage{amsthm, amssymb, amsmath, latexsym}
\usepackage{amsfonts}
\usepackage{graphicx}
\usepackage{color}
\usepackage[colorinlistoftodos]{todonotes}

\usepackage{mathabx}
\usepackage{breqn}

\theoremstyle{plain}
\begingroup
\newtheorem{theorem}{Theorem}[section] 
\newtheorem{lemma}[theorem]{Lemma}
\newtheorem{proposition}[theorem]{Proposition}
\newtheorem{corollary}[theorem]{Corollary}
\endgroup

\theoremstyle{definition}
\begingroup
\newtheorem{definition}[theorem]{Definition}
\newtheorem{remark}[theorem]{Remark}

\endgroup

\theoremstyle{remark}
\begingroup
\endgroup

\mathchardef\emptyset="001F

\numberwithin{equation}{section}

\newcommand{\Om}{\Omega}
\newcommand{\R}{\mathbb{R}}
\newcommand{\Rd}{\mathbb{R}^d}
\newcommand{\Rdd}{\mathbb{R}^{d\times d}}

\newcommand{\Ls}{{L^2}}

\newcommand{\Ht}{H^1(\Om\setminus \Gamma_t;\Rd)}
\newcommand{\V}{\mathbb V}
\newcommand{\C}{\mathbb C}

\newcommand{\A}{\mathbb A}

\newcommand{\dt}{\textnormal{d}t}
\newcommand{\ds}{\textnormal{d}s}
\newcommand{\dtau}{\textnormal{d}\tau}

\newcommand{\dive}{\textnormal{div}}

\newcommand{\e}{\varepsilon}

\newcommand\ee{\end{equation}}
\newcommand\be{\begin{equation}}

\title[Uniqueness and continuous dependence for a viscoelastic problem]
{Uniqueness and continuous dependence\\ for a viscoelastic problem with memory\\ in domains with time dependent cracks}

\author[Federico Cianci]{Federico Cianci}
\address[Federico Cianci]{SISSA, Via Bonomea 265, 34136 Trieste,
Italy}
\email[Federico Cianci]{fcianci@sissa.it}

\author[Gianni Dal Maso]{Gianni Dal Maso}
\address[Gianni Dal Maso]{SISSA, Via Bonomea 265, 34136 Trieste,
Italy}
\email[Gianni Dal Maso]{dalmaso@sissa.it}

\begin{document}

\begin{abstract}
We study some hyperbolic partial integro-differential systems in domains with time dependent cracks. In particular, we give conditions on the cracks which imply the uniqueness of the solution with prescribed initial-boundary conditions, and its continuous dependence on the cracks.
\end{abstract}

\maketitle


{\small{{\it 2020 MSC:\/}  
35Q74
, 74D05
, 74H20
.

}}

\section{Introduction}\label{intro}

The study of models of viscoelastic materials with memory has a long history that goes back to Boltzmann (\cite{Boltz_1} and \cite{Boltz_2}) and Volterra (\cite{Volterra_1} and \cite{Volterra_2}). Recent results on this subject can be found in \cite{DL_V1}, \cite{Fab-Gi-Pata}, \cite{Fab-Morro}, and \cite{Slepyan}. For particular values of the parameters, the Maxwell model for viscoelastic materials is governed by the following system of partial differential equations in $Q:=\Omega{\times}[0,T]$ with a memory term:
\begin{equation}\label{1.2}
    \ddot{u}(t) - \dive\big{(}(\C+\V) Eu(t)\big{)} + \dive\Big{(} \int^t_{0} \textnormal{e}^{\tau-t}\,\V Eu(\tau) \, \textnormal{d}\tau \Big{)} = \ell(t),
    \end{equation}
 where $\Omega \subset \Rd$ is the reference configuration, $[0, T]$ is the time interval, $u(t)$, $Eu(t)$, and $\ddot{u}(t)$ are the displacement at time $t$, the symmetric part of its gradient, and its second derivative with respect to time, $\C$ and $\V$ are the elasticity and viscosity tensors, and $\ell(t)$ is the external load at time $t$.
 
     In this paper we study problem \eqref{1.2} with a prescribed time dependent growing crack $\Gamma_t$, $t\in[0,T]$, namely
     \begin{equation}\label{intro-mainprobl}
    \ddot{u}(t) - \dive\big{(}(\C+\V) Eu(t)\big{)} + \dive\Big{(} \int^t_{0} \textnormal{e}^{\tau-t}\,\V Eu(\tau) \, \textnormal{d}\tau \Big{)} = \ell(t)  \,\,\, \text{ in } Q_{cr},
    \end{equation}
    where $Q_{cr}:=\{(x,t):t\in[0,T],\, x\in\Omega\setminus\Gamma_t\}$. Problem \eqref{intro-mainprobl} is complemented by initial conditions at $t=0$ for $u$ and $\dot u$ and by boundary conditions on $\partial\Omega$ and $\Gamma_t$.
    
    The existence of a solution of \eqref{intro-mainprobl} is proved in \cite{Sapio}. Our first result (Theorem \ref{thm:risultato-ex-uniq}) is the uniqueness of the solution under strong regularity assumptions on the sets $\Gamma_t$ and on their dependence on $t$. More precisely, we 
     assume the same regularity conditions that were used in \cite{DalMaso-Luc} and \cite{Caponi} to prove the uniqueness of the solution in $Q_{cr}$ of the problem without the memory term, i.e.,
     \begin{equation}\label{intro-onde}
    \ddot{u}(t) - \dive\big{(}(\C+\V) Eu(t)\big{)}  = \ell(t) \,\,\, \text{ in } Q_{cr}.
    \end{equation}
     
     To prove our uniqueness result we write problem \eqref{intro-mainprobl} in the equivalent form
     \begin{equation}\label{1.2-onde}
    \ddot{u}(t) - \dive\big{(}(\C+\V) Eu(t)\big{)}  = \ell(t) - \dive F_{u}(t) \,\,\, \text{ in } Q_{cr},
    \end{equation}
    where
    \begin{equation*}
        F_{u}(t):= \int^t_{0} \textnormal{e}^{\tau-t}\,\V Eu(\tau) \, \textnormal{d}\tau.
    \end{equation*}
    This allows us to estimate $u$ in terms of $F_u$ using the energy inequality for the solution of \eqref{intro-onde}. Then we estimate $F_u$ in terms of $u$ using just the definition of $F_u$. Uniqueness is obtained from the combined estimate.
    
    Our second result (Theorem \ref{thm:dip-cont-main}) is the continuous dependence of the solutions of \eqref{intro-mainprobl} on the cracks. More precisely, we consider a sequence $\Gamma^n_t$ of time dependent cracks and the solutions $u^n$ of problem \eqref{intro-mainprobl} with $\Gamma_t$ replaced by $\Gamma^n_t$. Under suitable assumption on the convergence of $\Gamma^n_t$ to $\Gamma_t$ we prove that the sequence $u^n$ converges to the solution $u$ of \eqref{intro-mainprobl}. Our assumptions of $\Gamma^n_t$ are similar to those considered in \cite{DalMaso-Luc} and \cite{Caponi} to prove the corresponding result for \eqref{intro-onde}.
    
    To prove the continuous dependence we write our problem in the form \eqref{1.2-onde} and we regard $u^n$ as a fixed point for a suitable operator depending on $n$, which is a contraction if $T$ is small enough. Under this assumption the convergence of $u^n$ is a consequence of a general results on fixed points of contractions (Lemma \ref{lemma:converg-punti-fissi}). To show that its hypotheses are satisfied, we use the continuous dependence on the cracks of the solutions of problem \eqref{intro-onde} (see \cite{DalMaso-Luc} and \cite{Caponi}) and we obtain the result if $T$ is small enough. If $T$ is large we divide the interval $[0,T]$ into smaller intervals where we can apply the previous result.

\section{Formulation of the problem}\label{Formulation}

The reference configuration of our problem
is a bounded open set $\Om\subset\Rd$, $d\geq 1$, with Lipschitz boundary $\partial\Om$. We assume that $\partial\Om=\partial_D\Om\cup \partial_N\Om$, where  $\partial_D\Om$ and $\partial_N\Om$ are disjoint (possibly empty) Borel sets, on which we prescribe
Dirichlet and Neumann boundary conditions respectively. 

 For every $x\in\overline\Omega$ the elasticity tensor $\C(x)$ and the viscosity tensor $\V(x)$ are prescribed elements of the space $\mathcal{L}(\R^{d\times d}_{sym}; \R^{d\times d}_{sym})$ of linear maps from $\R^{d\times d}_{sym}$ into $\R^{d\times d}_{sym}$, where $\Rdd_{sym}$ is the space of reald $d\times d$ symmetric matrices. The euclidean scalar product between the matrices $A$ and $B$ is denoted by $A:B$. We assume that the functions $\C,\, \V \colon \overline{\Om} \to \mathcal{L}(\R^{d\times d}_{sym}; \R^{d\times d}_{sym})$ satisfy the following properties, for suitable constants $\alpha_0>0$ and $M_0>0$:
\begin{itemize}
    \item[(H1)] (regularity) $\C$ is of class $C^1$ and $\max_{x \in \overline\Omega} |\C(x)|\leq M_0$;\smallskip
    \item[(H2)] (symmetry) $ \C(x)A : B= A: \C(x) B$ for every $x\in\overline\Om$ and $A,\,B\in \R^{d\times d}_{sym}$;\smallskip
    \item[(H3)] (coerciveness) $ \C(x)A : A\geq \alpha_0 |A|^2$ for every $x\in\overline\Om$ and $A\in \R^{d\times d}_{sym}$;\smallskip
    \item[(H4)] (regularity) $\V$ is of class $C^1$ and $\max_{x \in \overline\Omega} |\V(x)|\leq M_0$;\smallskip
    \item[(H5)] (symmetry) $\V(x)A : B=A :\V(x)  B$ for every $x\in\Om$ and $A,\,B\in \R^{d\times d}_{sym}$;\smallskip
    \item[(H6)] (coerciveness) $ \V(x)A : A\geq \alpha_0 |A|^2$ for every $x\in\overline\Om$ and $A\in \R^{d\times d}_{sym}$.
\end{itemize}
 
 Throughout the paper we study the problem in the time interval $[0,T]$, with $T>0$. For $t\in[0,T]$ the crack at time $t$ is given by a subset $\Gamma_t$ of the intersection between $\overline\Omega$ and a suitable $d-1$ dimensional manifold $\Gamma$ (regarded as the crack path). We assume that
\begin{itemize}
\item[(H7)] $\Gamma$ is a complete $(d-1)$-dimensional $C^2$ manifold with boundary;\smallskip
\item[(H8)] $\Omega \cap \partial \Gamma = \emptyset $ and $ \mathcal{H}^{d-1}( \Gamma \cap \partial\Om{} )=0$, where $\mathcal{H}^{d-1}$ denotes the $(d-1)$-dimensianal Hausdorff measure;\smallskip
\item[(H9)] for every $x\in \Gamma\cap \partial\Omega$ there exists an open neighborhood $U_x$ of $x$ in $ \Rd$
such that $U_x \cap (\Omega\setminus \Gamma)$
is the union of two non empty disjoint open sets $U^+_x$ and $U^-_x$ with Lipschitz boundary;\smallskip
\item[(H10)] $\Gamma_t$ is closed, $\Gamma_t\subset \Gamma \cap \overline{\Omega}$ for every $t\in [0,T]$, and $\Gamma_s \subset \Gamma_t$ for every $s< t$ (irreversibility of the fracture process).
\end{itemize}

Moreover we assume that there exist $\Phi, \Psi: [0, T] \times \overline{\Om} \to \overline{\Om}$ with the following properties:
\begin{itemize}
\item[(H11)] $\Phi, \Psi$ are of class $C^{2,1}$;\smallskip
\item[(H12)]  $\Psi(t, \Phi(t, y)) = y$ and $\Phi(t, \Psi(t, x)) = x$ for every $x, y \in \overline{\Om}$;\smallskip
\item[(H13)] $\Phi(t, \Gamma) = \Gamma,\, \Phi(t, \Gamma_0) = \Gamma_t,$ and $ \Phi(t, y) = y$ for every $t \in [0, T]$ and every $y$ in a
neighborhood of $\partial \Omega$;\smallskip
\item[(H14)] $\Phi(0,y)=y$ for every $y\in \overline\Omega$;\smallskip
\item[(H15)] $|\dot\Phi(t,y)|^2<\frac{m_{det}(\Psi)\alpha_0}{M_{det}(\Psi)K}$ for every $y\in \overline\Omega$, where the dot denotes the derivative with respect to $t$,   $m_{det}(\Psi):=\min \det D\Psi$, $M_{det}(\Psi):=\max \det D\Psi$. and $K$ is the constant in Korn’s inequality in Lemma \ref{rem:Korn_ineq} below.
\end{itemize}

We shall prove that our hypotheses imply that Korn’s inequality holds on $\Omega\setminus \Gamma$. We begin with the following technical lemma.

\begin{lemma}\label{lemma:aperti-lip}
Under hypotheses {\rm(H7)-(H9)}, the set $\Omega\setminus \Gamma$ is the union of a finite number of connected open sets with Lipschitz boundary.
\end{lemma}
\begin{proof}
Since $\Gamma$ is a $C^2$ manifold of dimension $d-1$, for every $x\in \Gamma \cap \Omega $ there exists an open neighborhood $U_x$ of $x$ in $ \Rd$
such that $U_x \cap (\Omega\setminus \Gamma)$
is the union of two non empty disjoint open sets $U^+_x$ and $U^-_x$ with Lipschitz boundary. By our hypothesis on $\Gamma \cap \partial \Omega$ the same property holds, more in general, for every $x\in\Gamma \cap \overline{\Omega}$. Since $\Gamma \cap \overline{\Omega}$ is compact, there exists a finite number of points $x_1,...,x_m\in \Gamma \cap \overline{\Omega}$ such that
$\Gamma \cap \overline{\Omega} \subset \cup^m_{i=1} U_{x_i}$.

Since $\Omega$ has Lipschitz boundary, for every $y\in\partial\Omega\setminus\cup^m_{i=1}U_{x_i}\subset \partial\Omega\setminus\Gamma$ there exists an open neighborhood $V_y$ of $y$ in $\Rd$ such that $V_y\cap (\Omega\setminus\Gamma)$ has Lipschitz boundary. By compactness there exists a finite number of points $y_1,...,y_n\in \partial\Omega\setminus\cup^m_{i=1}U_{x_i} $ such that $\partial\Omega\setminus\cup^m_{i=1}U_{x_i}\subset \cup^n_{j=1}V_{y_j}$.

Since $\overline{\Omega}\setminus(\cup^m_{i=1}U_{x_i}\cup \cup^n_{j=1}V_{y_j})$ is compact and is contained in the open set $\Omega\setminus\Gamma$, there exists an open set $W$ with Lipschitz boundary such that $\overline{\Omega}\setminus(\cup^m_{i=1}U_{x_i}\cup \cup^n_{j=1}V_{y_j})\subset W \subset \Omega \setminus \Gamma $.
Therefore
\begin{equation*}
    \Omega \setminus \Gamma = W \cup \bigcup^m_{i=1}U^+_{x_i} \cup \bigcup^m_{i=1}U^-_{x_i} \cup \bigcup^n_{j=1}(V_{y_j}\cap(\Omega\setminus\Gamma)).
\end{equation*}
Since every open sets with Lipschitz boundary is the union of a finite number of connected open sets with Lipschitz boundary, the conclusion follows.
\end{proof}

For every $u\in H^1(\Omega\setminus\Gamma;\Rd)$ $Du$ denotes jacobian matrix in the sense of distributions on $\Omega\setminus\Gamma$ and $Eu$ is its symmetric part, i.e.,
\begin{equation*}
    Eu:=\tfrac{1}{2}(Du + Du^T).
\end{equation*}

\begin{lemma}\label{rem:Korn_ineq}
 Under hypotheses {\rm(H7)-(H9)}, there exists a constant $K$, depending only on $\Om$ and $\Gamma$, such that
\begin{equation}\label{eq:equiv-norm}
\|Du\|^2 \leq K ( \|u\|^2 + \|Eu\|^2 ) 
\end{equation}
for every $u\in H^1(\Omega\setminus\Gamma;\Rd)$, where $\|\cdot\|$ denotes the $L^2$ norm.
\end{lemma}
\begin{proof}
The result is a consequence of the second Korn’s inequality (see, e.g., \cite[Theorem 2.4]{Ol-Sha-Yos}), applied to the sets with Lipschitz boundary provided by Lemma \ref{lemma:aperti-lip}.
\end{proof}

\begin{remark}
 Under hypotheses {\rm(H7)-(H9)}, using a localization argument (see the proof of Lemma \ref{lemma:aperti-lip}) we can prove that the trace operator is well defined and continuous from $H^1(\Omega\setminus\Gamma;\Rd)$ into $L^2(\partial\Omega;\Rd)$.
\end{remark}

We now introduce the function spaces that will be used in the precise formulation of problem \eqref{intro-mainprobl}. We set
\begin{equation}\label{eq:VeH}
    V:=H^1(\Omega\setminus\Gamma;\Rd)\quad\textnormal{and}\quad H:=L^2(\Omega;\Rd{}).
\end{equation}
For every finite dimensional Hilbert space $Y$ the symbols $(\cdot\,,\cdot)$ and $\|\cdot\|$ denote the scalar product and the norm in the $L^2(\Omega; Y)$, according to the context.
The space $V$ is endowed with the norm
\begin{equation}
    \|u\|_V:=\big{(}\|u\|^2+ \|Du\|^2\big{)}^{{1}/{2}}.
\end{equation}
For every $t\in [0,T]$ we define
\begin{equation}\label{eq:def-spazi}
V_t:=\Ht \quad \textnormal{and} \quad V^D_t:=\big\{u\in V_t\,\,\big|\,\,u|_{\partial_D\Om}=0\big\},
\end{equation}
where $u|_{\partial_D\Om}$ denotes the trace of $u$ on $\partial_D\Omega$. We note that $V_t$ and $V^D_t$ are closed linear subspaces of $V$.

We define
\begin{equation}\label{eq:Vcorsivo}
    \mathcal{V}:=\big\{ v\in \Ls(0,T; V)\cap H^1(0,T;H)\,\big|\,v(t)\in V_t\,\text{  for a.e. }\,t\in (0,T) \big\},
\end{equation}
which is a Hilbert space with the norm
\begin{equation}\label{eq:norma_U}
    \|v\|_{\mathcal{V}}:=\big{(}\|v\|^2_{\Ls(0,T;V)}+\|\dot v\|^2_{\Ls(0,T;H)}\big{)}^{\frac{1}{2}},
\end{equation}
where the dot denotes the distibutional derivative with respect to $t$. Moreover we set
\begin{equation}\label{eq:VcorsivoD}
     \mathcal{V}^D:=\big\{ v\in \mathcal{V} \,\big|\,v(t)\in V^D_t\,\text{  for a.e. }\,t\in (0,T) \big\}
\end{equation}
and note that it is a closed linear subspace of $\mathcal{V}$. Since $H^1(0,T;H)\hookrightarrow C^0([0,T];H)$ we have $\mathcal{V}\hookrightarrow C^0([0,T],H)$. In particular $v(0)$ and $v(T)$ are well defined as elements of $H$, for every $v\in \mathcal{V}$.

We set
\begin{equation}
    \tilde{H}:=L^2(\Omega;\R^{d\times d}_{sym}).
\end{equation}
On the forcing term $\ell(t)$ of \eqref{intro-mainprobl} we assume that
\begin{equation}
    \ell(t):= f(t) - \dive F(t),
\end{equation}
where
\begin{equation}\label{eq:forzante-f}
    f\in \Ls(0,T;H) \quad \text{and} \quad
    F\in H^1(0,T;\tilde{H})
\end{equation}
are prescribed function. As usual the divergence of a matrix valued function is the vector valued function whose components are obtained taking the divergence of the rows.

As for the Dirichlet boundary condition on $\partial_D \Omega$, it is obtained by prescribing a function
\begin{equation}\label{eq:dato-u_D}
    u_D\in H^2(0,T;\,H)\cap H^1(0,T;\,V_0).
\end{equation}
We impose that for a.e. $t\in[0,T]$ the trace of the solution $u(t)$ is equal to the trace $u_D(t)$ on $\partial_D\Om$, i.e., $u(t)-u_D(t) \in {V}^D_t$.

About the initial data we fix
\begin{equation}\label{eq:dati-in-u0u1}
    u^0 \in V_0 \quad \text{and}\quad u^1\in H.
\end{equation}
Moreover, we assume the compatibility condition
\begin{equation}\label{eq:comp-conditt}
    u^0-u_D(0)\in V^D_0.
\end{equation}

We are now in a position to give the precise definition of solution of problem \eqref{intro-mainprobl}. 

\begin{definition}[Solution for visco-elastodynamics with cracks]\label{def:weak-sol}
We say that $u$ is a weak solution of problem \eqref{intro-mainprobl} of visco-elastodynamics on the cracked domains $\Omega\setminus\Gamma_t$, with external load $\ell=f-\dive F$,  Dirichlet boundary condition $u_D$ on $\partial_D\Omega$, natural Neumann boundary condition on $\partial_N\Omega\cup\Gamma_t$, and initial conditions $u^0$ and $u^1$, if
\begin{align}
    &\label{eq:weak_generalizzata1} \,\,\,u\in\mathcal{V} \quad \text{and} \quad u-u_D\in \mathcal{V}^D,\\
    &-\int^T_0(\dot u(t),\dot \varphi(t))\,\dt+ \int^T_0 ((\C+ \V) Eu(t),E\varphi(t)) \,\dt - \int^T_0\int^t_0  \textnormal{e}^{\tau-t} (\V Eu(\tau), E\varphi(t))\,\dtau\dt\vspace{1cm}\nonumber \\
    &\label{eq:weak_generalizzata2}= \int^T_0 (f(t),\varphi(t))\,\dt{} + \int^T_0  (F(t),E\varphi(t))  \,\dt{}  \quad\text{for all }  \varphi\in \mathcal{V}^D  \text{ with } \varphi(0)=\varphi(T)=0,\\
    & \label{eq:weak_generalizzata3}\,\,\,u(0)=u^0 \quad \text{in }  H \quad  \text{and} \quad \dot u(0)=u^1 \quad \text{in } (V^D_0)^*,
\end{align}
where $(V^D_0)^*$ denotes the topological dual of $V^D_t$ for $t=0$.
\end{definition}
\begin{remark}\label{remark 3.5}
 If $u$ satisfy \eqref{eq:weak_generalizzata1} and \eqref{eq:weak_generalizzata2}, it is possible to prove that $\dot u\in H^1(0,T;(V^D_0)^*)$ (see \cite[Remark 4.6]{Sapio}), which implies $\dot u\in C^0([0,T];(V^D_0)^*)$. In particular $\dot u(0)$ is well defined as an element of $(V^D_0)^*$.
\end{remark}

\begin{remark}
 Under suitable regularity assumptions, $u$ is a solution in the sense of Definition \ref{def:weak-sol} if and only if $u(0)=u^0$, $\dot u(0)=u^1$, and for every $t\in[0,T]$
\begin{alignat*}{2}
    & \ddot{u}(t) - \dive\big{(}(\C + \V) Eu(t)\big{)} + \dive\Big{(} \int^t_0 \textnormal{e}^{\tau-t}\,\V Eu(\tau) \, \textnormal{d}\tau \Big{)} = f(t)-\dive F(t) \qquad &&\textnormal{in } \Om\setminus\Gamma_t,\\
    &u(t) = u_D(t) \qquad && \text{on $\partial_D\Om$,} \\
    &\Big{(}(\C + \V) Eu(t)  -  \int^t_0 \textnormal{e}^{\tau-t}\,\V Eu(\tau) \, \textnormal{d}\tau \Big{)}\nu=F(t)\nu\qquad && \text{on $\partial_N\Om$,}\\
    &\Big{(}(\C + \V) Eu(t)  -  \int^t_0 \textnormal{e}^{\tau-t}\,\V Eu(\tau) \, \textnormal{d}\tau \Big{)}^{\!\pm}\!\!\!\nu=F(t)^\pm\nu\qquad && \text{on $\Gamma_t$,}\label{eq:bcNCrack}
    \end{alignat*}
    where $\nu$ is the unit normal and the symbol $\pm$ denotes suitable limits on each side of $\Gamma_t$.
    
    The last two conditions represent the natural Neumann boundary conditions on $\partial_N\Omega$ and on the faces of $\Gamma_t$.
\end{remark}

To describe the boundedness properties of the solutions of problem \eqref{eq:weak_generalizzata1}-\eqref{eq:weak_generalizzata3}, we introduce the space
\begin{equation}\label{eq:defVinfty}
    \mathcal{V}^\infty:=\big\{ v\in L^\infty(0,T; V)\cap W^{1,\infty}(0,T;H)\,\big|\,v(t)\in V_t\,\text{  for a.e. }\,t\in (0,T) \big\},
\end{equation}
which is a Banach space with the norm
\begin{equation}\label{eq:defnormaVinfty}
    \|v\|_{\mathcal{V^\infty}}:=\|v\|_{L^\infty(0,T;V)}+\|\dot v\|_{L^\infty(0,T;H)}.
\end{equation}
As for the continuity properties, it is convenient to introduce the space of weakly continuous functions with values in a Banach space $X$ with topological dual $X^*$, defined by
\begin{equation*}
    C^0_w([0,T]; X):=\big\{ v: [0,T] \to X \,\big|\, t \mapsto \langle h,\, v(t) \rangle \textnormal{ is continuous for every } h \in X^* \big\}.
\end{equation*}
We are now in position to state one of the main results of the paper.
\begin{theorem}\label{thm:risultato-ex-uniq}
Assume {\rm(H1)-(H15)} and \eqref{eq:forzante-f}-\eqref{eq:comp-conditt}. Then there exists a unique solution of problem \eqref{eq:weak_generalizzata1}-\eqref{eq:weak_generalizzata3}. Moreover $u \in \mathcal{V}^\infty$, $u\in C^0_w([0,T];V)$, and $\dot u\in C^0_w([0,T];H)$.
\end{theorem}

The existence of a solution is proved in \cite{Sapio} under much weaker assumptions on the cracks $\Gamma_t$. The uniqueness will be proved in the next section.

\section{Uniqueness}\label{Wave}

In our proof of Theorem \ref{thm:risultato-ex-uniq} we shall use some known results about existence and uniqueness for the system of elastodynamics on cracked domains, where the memory terms is not present. We set
\begin{equation}
    \A:=\C+\V
\end{equation}
and we consider $\A$ as the elasticity tensor of the auxiliary problem defined below.

\begin{definition}[Solution for elastodynamics with cracks]\label{def:weak-sol_wave}
We say that $v$ is a weak solution of problem \eqref{intro-onde} of elastodynamics on the cracked domains $\Omega\setminus\Gamma_t$, with external load $\ell=f-\dive F$,  Dirichlet boundary condition $u_D$ on $\partial_D\Omega$, natural Neumann boundary condition on $\partial_N\Omega\cup\Gamma_t$, and initial conditions $u^0$ and $u^1$, if
\begin{align}
    &\label{eq:onde1} \,\,\,v\in\mathcal{V} \quad \text{and} \quad v-u_D\in \mathcal{V}^D,\\
    &-\int^T_0(\dot v(t),\dot \varphi(t))\,\dt+ \int^T_0 (\A Ev(t),E\varphi(t)) \,\dt = \int^T_0 (f(t),\varphi(t))\,\dt{}  \vspace{1cm}\nonumber \\
    &\label{eq:onde2}  + \int^T_0  (F(t),E\varphi(t))  \,\dt{} \quad \text{for all }  \varphi\in \mathcal{V}^D  \text{ with } \varphi(0)=\varphi(T)=0,\\
    & \label{eq:onde3}\,\,\,v(0)=u^0 \quad \text{in }  H \quad  \text{and} \quad \dot v(0)=u^1 \quad \text{in } (V^D_0)^*.
\end{align}
\end{definition}

The following technical lemma will be used in the proof of Theorem \ref{thm:ex-uniq-onde}.
\begin{lemma}\label{prop:dati-iniz-nulli}
    Let $v$ be a weak solution according to Definition \ref{def:weak-sol_wave} satisfying $\dot v(0)=0$ in the sense of $(V^D_0)^*$. Then \eqref{eq:onde2} holds for every $\varphi\in \mathcal{V}^D$ such that $\varphi(0)\in V^D_0$ and $\varphi(t)=0$ in a neighborhood of $T$, even if the condition $\varphi(0)=0$ is not satisfied.
\end{lemma}
\begin{proof}
Let $\varphi$ as in the statement. For every $\varepsilon>0$, we define
\begin{equation*}
   \varphi_\varepsilon(t):= \begin{cases}
    \frac{t}{\varepsilon} \varphi(0) & \textnormal{for }t\in [0,\, \varepsilon],\\
    \varphi(t-\varepsilon) & \textnormal{for } t \in (\varepsilon, \, T ].
    \end{cases}
\end{equation*}
Then $\varphi_\varepsilon\in \mathcal{V}^D$ and $\varphi_\varepsilon(0)=\varphi_\varepsilon(T)=0$, for $\varepsilon$ small enough \eqref{eq:onde2} holds for $\varphi_\e$. We observe that
\begin{equation*}
    \int^T_0(\dot v(t),\dot \varphi_\e(t))\,\dt = \frac{1}{\e}\int^\e_0(\dot v(t), \varphi(0))\,\dt  + \int^T_\e(\dot v(t),\dot \varphi(t-\e))\,\dt \to  \int^T_0 (\dot v(t),\dot \varphi(t))\,\dt
\end{equation*}
as $\e \to 0$, where we have used the initial condition in the first term and the continuity of translations in the second one. In a similar way we can pass to the limit in the other terms of equation \eqref{eq:onde2}.
\end{proof}

We are now in a position to state the existence and uniqueness result for the solutions of elastodynamics with cracks.

\begin{theorem}\label{thm:ex-uniq-onde}
Assume {\rm(H1)-(H15)} and \eqref{eq:forzante-f}-\eqref{eq:comp-conditt}. Then there exists a unique solution $v$ of problem \eqref{eq:onde1}-\eqref{eq:onde3}. Moreover $ v \in \mathcal{V}^\infty$, $v\in C^0_w([0,T];V)$, and $\dot v\in C^0_w([0,T];H)$.
\end{theorem}

\begin{proof}
In the case $F=0$ the existence result, together with an energy bound, is proved in \cite{Caponi} and \cite{Tasso} (a previous result in the scalar case is proved in \cite{Dalmaso-Larsen}). When $F$ is present, the same proof can be repeated with obvious modifications (for instance it is enough to repeat the arguments of \cite{Sapio} with $\V=0$). 

As for uniqueness, it can be proved as in \cite[Example 4.2 and Theorem 4.3]{DalMaso-Toader}. Since in that paper the initial conditions are given in a different sense, we have to replace \cite[Proposition 2.10]{DalMaso-Toader} by our Lemma \ref{prop:dati-iniz-nulli}. The uniqueness result and the existence of a solution with bounded energy imply that the solution satisfies $ v \in \mathcal{V}^\infty$. This fact, together with the continuity of $v$ in $H$ and $\dot v\in (V_0^D)^*$ (Remark~\ref{remark 3.5}), implies that $v\in C^0_w([0,T];V)$ and $\dot v\in C^0_w([0,T];H)$ (see, e.g., \cite[Chapitre XVIII, §5, Lemme 6]{D-L_V8}).
\end{proof}
For every $v\in  C^0_w([0,T];V)$, with $\dot v \in C^0_w([0,T];H)$, the energy of $v$ is defined for every $ t \in [0,T]$ as
\begin{equation}
\mathcal{E}_v(t):=\frac{1}{2}\|\dot{v}(t)\|^2 +  \frac{1}{2}(\A Ev(t),Ev(t)).
\end{equation}
Under the same assumption on $v$, when $u_D=0$ the work done by the external forces on the displacement $v$ in the time interval $[0,t] \subset [0,T]$ can be written as
\begin{equation}
    \mathcal{W}_v(t):= \int^t_0  ({f}(s), \dot{v}(s) )  \, \ds - \int^t_0 (\dot F(s), Ev(s))\, \ds + ( F(t), Ev(t)) - ( F(0), Ev(0)),
\end{equation}
see for instance \cite[Remarks 5.9 and 5.11]{Sapio}.
\begin{theorem}
Under the assumptions of Theorem \ref{thm:ex-uniq-onde}, if $u_D=0$, then the unique solution $v$ of problem \eqref{eq:onde1}-\eqref{eq:onde3} satisfies the energy inequality
\begin{equation}\label{eq:energy-bal-wave}
    \mathcal{E}_v(t)\leq \mathcal{E}_v(0) + \mathcal{W}_v(t) \quad \text{for all }\,t\in [0,T].
\end{equation}
\end{theorem}
For a proof we refer to \cite[Corollary 3.2]{DalMaso-Toader} and \cite[Remark 5.11]{Sapio}.

\begin{proposition}\label{prop:energ-soloF}
    Under the assumptions of Theorem \ref{thm:ex-uniq-onde}, suppose in addition that $u_D=0$ and $u^0=0$. Then there exists a positive constants $A$, depending on the constant $K$ in Korn's inequality \eqref{eq:equiv-norm} and on the constant $\alpha_0$ in (H3), but not on $T$, $f$, $F$, and $u^1$, such that the solution $v$ of problem \eqref{eq:onde1}-\eqref{eq:onde3} satisfies
    \begin{equation}\label{eq:cont-ris-soloF}
        \|v\|_{\mathcal{V}^\infty}\leq A(1+T)\Big(\|u^1\|+\|F \|_{L^\infty(0,T;\tilde{H})}+ T^{1/2}(\|\dot F \|_{L^2(0,T;\tilde{H})}+\|f \|_{L^2(0,T;{H})})\Big).
    \end{equation}
\end{proposition}
\begin{proof}
Under our assumption we have
\begin{equation*}
    \mathcal{W}_v(t):= \int^t_0 ({f}(s), \dot{v}(s)) \ds - \int^t_0 (\dot F(s), Ev(s))\, \ds + ( F(t), Ev(t))\quad \text{and} \quad \mathcal{E}_v(0)=\frac{1}{2}\|u^1\|^2.
\end{equation*}
Recalling (H3), (H6), and \eqref{eq:energy-bal-wave} we have
\begin{align*} 
\frac{1}{2}\|\dot v(t)\|^2 + \frac{\alpha_0}{2}\|Ev(t)\|^2 &\leq T^{1/2}\|\dot F \|_{L^2(0,T;\tilde{H})}\|Ev\|_{L^\infty(0,T;\tilde{H})} + \|F \|_{L^\infty(0,T;\tilde{H})}\|Ev\|_{L^\infty(0,T;\tilde{H})}\\
&+ T^{1/2}\|f\|_{L^2(0,T;H)}\|\dot{v}\|_{L^\infty(0,T;H)} + \frac{1}{2}\|u^1\|^2.
\end{align*}
for all $t\in [0,T]$. We set
\begin{equation*}
    S:=\sup_{t\in [0,T]} (\|\dot v(t)\|^2 + \|Ev(t)\|^2)^{1/2}.
\end{equation*}
From the previous inequality we obtain
\begin{equation}\label{contrS}
    \min\{1/2,\,\alpha_0/2\}S \leq T^{1/2}\|\dot F \|_{L^2(0,T;\tilde{H})} + \|F \|_{L^\infty(0,T;\tilde{H})} + T^{1/2}\|f\|_{L^2(0,T;H)} + \|u^1\|.
\end{equation}
Since $v(t)=\int^t_0 \dot v(s) \,\ds$ we have
$
\sup_{t\in[0,T]}\|v(t)\|\leq TS
$. Using Korn's inequality \eqref{eq:equiv-norm} we obtain
$
\sup_{t\in[0,T]}\| Dv(t)\|\leq K^{1/2}S. 
$
Therefore
\begin{equation*}
    \|u\|_{\mathcal{V}^\infty} \leq S + K^{ 1/2}S + TS,
\end{equation*}
which, together with \eqref{contrS}, gives \eqref{eq:cont-ris-soloF}.
\end{proof}

Let $\mathcal{L}:\mathcal{V}^\infty\longrightarrow H^1(0,T; \tilde{H}) $ be the linear operator defined by
\begin{equation}
    (\mathcal{L}u)(t):=\int^t_0  \textnormal{e}^{\tau-t} \V Eu(\tau) \,\dtau
\end{equation}
for every $u\in \mathcal{V}^\infty$ and $t\in [0,T]$. Since
\begin{equation*}
     (\dot{\wideparen{\mathcal{L}u}})(t)=  \V Eu(t) -  \int^t_0  \textnormal{e}^{\tau-t} \V Eu(\tau) \,\dtau,
\end{equation*}
it is easy to check that $\mathcal{L}$ is bounded. Indeed we have
\begin{equation}\label{eq:Lu-inf}
    \|\mathcal{L}u\|_{L^\infty(0,T;\tilde{H})} \leq T \|\V\|_{\infty} \| u \|_{\mathcal{V}^\infty},
\end{equation}
\begin{equation}\label{eq:dotLu-inf}
     \|\dot{\wideparen{\mathcal{L}u}}\|_{L^2(0,T;\tilde{H})}\leq(T^{1/2}+T^{3/2})\|\V\|_{\infty} \| u \|_{\mathcal{V}^\infty}.
\end{equation}

\begin{corollary}\label{cor:stima-ris-datinulli}
 Under the assumptions of Theorem \ref{thm:ex-uniq-onde} there exists a positive constant $B$, depending on the constant $K$ in Korn's inequality \eqref{eq:equiv-norm} and on the constant $\alpha_0$ in (H3), but not on $T$ and $\V$, such that, if $u$ satisfies \eqref{eq:onde1}-\eqref{eq:onde3} with $u^0$, $u^1$, $u_D$, and $f$ replaced by zero and $F$ replaced by $\mathcal{L}u$, then
 \begin{equation}\label{eq:contris-soloLu}
        \|u\|_{\mathcal{V}^\infty}\leq B (T+ T^3)\|\V\|_{\infty} \| u \|_{\mathcal{V}^\infty}.
    \end{equation}
\end{corollary}
\begin{proof}
By Proposition \ref{prop:energ-soloF}, \eqref{eq:Lu-inf}, and \eqref{eq:dotLu-inf} we have
\begin{equation*}
        \|u\|_{\mathcal{V}^\infty}\leq A\Big((1+T)T + (T^{1/2}+ T^{3/2})^2\Big)\|\V\|_{\infty} \| u \|_{\mathcal{V}^\infty},
    \end{equation*}
    which implies \eqref{eq:contris-soloLu}.
\end{proof}

We are now in a position to prove the uniqueness result.

\begin{proof}[Proof of Theorem \ref{thm:risultato-ex-uniq}]
The existence result is obtained in [11] under more general hypotheses. To prove uniqueness, we assume by contradiction that there exist two distinct solution $u_1$ and $u_2$ of problem \eqref{eq:weak_generalizzata1}-\eqref{eq:weak_generalizzata3}. Then $u:=u_1-u_2$ is a solution of the same problem with $u^0$, $u^1$, $u_D$, $f$, and $F$ replaced by zero. Therefore $u$ satisfies \eqref{eq:onde1}-\eqref{eq:onde3} with 
$u^0$, $u^1$, $u_D$, and $f$ replaced by zero
and $F$ replaced by $\mathcal{L}u$. By Theorem \ref{thm:ex-uniq-onde} this implies that $u\in C_w([0,T]; V)$ and $\dot u\in C_w([0,T]; H)$.

We set
\begin{equation*}
    t_0:=\inf\{t\in [0,T]\,|\,u(t)\neq 0\}.
\end{equation*}
Since $u$ is not identically zero, we have $t_0<T$. We fix $\delta\in (0,T-t_0)$ such that
\begin{equation}\label{eq:condT-delta}
    B( \delta +\delta^{3})\|\V\|_{\infty} < 1,
\end{equation}
where $B$ is the constant in \eqref{eq:contris-soloLu}, and we define $t_1:=t_0+\delta$. In order to study the problem on $[t_0,\,t_1]$ we define the spaces $\mathcal{V}^D_{t_0,t_1}$ and $\mathcal{V}^\infty_{t_0,t_1}$ as $\mathcal{V}^D$ and $\mathcal{V}^\infty$ (see \eqref{eq:VcorsivoD} and \eqref{eq:defVinfty}), with $0$ and $T$ replaced by $t_0$ and $t_1$. 

It is clear that $u \in \mathcal{V}^D_{t_0,t_1}$ and since $Eu(\tau)=0$ for every $\tau \in [0,t_0]$ we have
\begin{align*}
    &-\int^{t_1}_{t_0}(\dot u(t),\dot \varphi(t))\,\dt+ \int^{t_1}_{t_0} (\A Eu(t),E\varphi(t)) \,\dt - \int^{t_1}_{t_0}\int^t_{t_0} \textnormal{e}^{\tau-t} (\V Eu(\tau), E\varphi(t))\,\dtau\dt=0\vspace{1cm}\nonumber 
\end{align*}
for every $\varphi\in \mathcal{V}^D_{t_0,t_1}$ such that $\varphi({t_0})=\varphi({t_1})=0$. Moreover, since $u\in C_w([0,T]; V)$, $\dot u\in C_w([0,T]; H)$, and $u$ is identically zero on $[0,t_0]$, we have that  $u(t_0)=0$ and $\dot u(t_0)=0$. By \eqref{eq:contris-soloLu}, applied with $0$ and $T$ replaced by $t_0$ and $t_1$, we have
\begin{equation*}
        \|u\|_{\mathcal{V}^\infty_{t_0,t_1}}\leq B( \delta +\delta^{3})\|\V\|_{\infty}   \|u \|_{\mathcal{V}^\infty_{t_0,t_1}}.
\end{equation*}
Using \eqref{eq:condT-delta} we obtain $u=0$ on $[t_0,t_1]$. This contradicts the definition of $t_0$ and concludes the proof.
\end{proof}

\section{Continuous dependence on the data}

In this section we consider a sequence $\{\Gamma^n_t\}_{t\in[0,T]}$ of time dependent cracks and we want to study the convergence, as $n\to+\infty$, of the solutions of the corresponding viscoelastic problems. For completeness we assume that also the other data of the problem depend on~$n$.

For every $n \in \mathbb{N}$, let $\C^n,\, \V^n \colon \overline{\Om} \to \mathcal{L}(\R^{d\times d}_{sym}; \R^{d\times d}_{sym})$, let $\Gamma^n$ be a $(d - 1)$-dimensional $C^2$ manifold, let $\{\Gamma^n_t\}_{t\in [0,T]}$ be a family of closed subsets of $\Gamma^n$, and let $\Phi^n,\,\Psi^n\colon[0,T]\times \overline{\Omega}\to\overline{\Omega}$. We assume that
\begin{itemize}
    \item[(H16)] $\C^n,\, \V^n$ satisfy  (H1)-(H6) with constants $\alpha_0$ and $M_0$ independent of $n$;\smallskip
    \item[(H17)] $\Gamma^n$ and $\{\Gamma^n_t\}_{t\in [0,T]}$ satisfy (H7)-(H10);\smallskip
    \item[(H18)] $\Phi^n,\,\Psi^n$ satisfy (H11)-(H15) (with $\Gamma$ and $\Gamma_t$ replaced by $\Gamma^n$ and $\Gamma^n_t$), the latter with the constant $K$ that appears in \eqref{eq:korn-unif}.
\end{itemize}
Let $\Rdd$ be the space of $d\times d$ real matrices. For every pair of normed spaces $X$ and $Y$ let $\mathcal{L}(X;Y)$ be the space of linear and continuous maps between $X$ and $Y$. For every $x\in \overline{\Omega}$ it is convenient to consider the extensions $\C_e(x),\,\V_e(x),\,\C^n_e(x),\,\V^n_e(x)\in \mathcal{L}(\Rdd{};\Rdd_{sym})$ of the linear maps $\C(x),\,\V(x),\,\C^n(x),\,\V^n(x)$ defined as
\begin{equation}\label{eq:tensori-estesi}
\C^n_e(x)[A]:=\C^n(x)[A_{sym}] \quad\text{and} \quad \V^n_e(x)[A]:=\V^n(x)[A_{sym}] \quad \text{for all } A \in \Rdd{},
\end{equation}
\begin{equation}\label{eq:tensori-estesi-limite}
\C_e(x)[A]:=\C(x)[A_{sym}] \quad\text{and} \quad \V_e(x)[A]:=\V(x)[A_{sym}] \quad \text{for all } A \in \Rdd{},
\end{equation}
where $A_{sym}$ is the symmetric part of the matrix $A$. Moreover we set
\begin{equation}\label{eq:A-An}
    \mathbb{A}^n_e:=\C^n_e+\V^n_e \quad \text{and} \quad
    \mathbb{A}_e:=\C_e+\V_e.
\end{equation}

For technical reasons we use a change of variable which maps $\Gamma^n_0$ into $\Gamma_0$. This is done by means of diffeomorphisms $\Theta^n,\,\Xi^n\colon\overline{\Omega} \to \overline{\Omega}$ such that
\begin{itemize}
\item[(H19)] $\Theta^n$ and $\Xi^n$ are of class $C^{2,1}$;\smallskip
\item[(H20)]  $\Theta^n( \Xi^n( x)) = x$ and $\Xi^n( \Theta^n(x)) = x$ for every $x\in \overline{\Om}$;\smallskip
\item[(H21)] $\det D\Theta^n(x)>0$ for every $x \in \overline{\Om{}}$;\smallskip
\item[(H22)] $\Theta^n(\Gamma \cap \overline{\Omega}) = \Gamma^n \cap \overline{\Omega},$ and $ \Theta^n(\Gamma_0) = \Gamma^n_0$;\smallskip
\item[(H23)] $ \Theta^n(\partial_D\Omega) = \partial_D\Omega$ and $ \Theta^n(\partial_N\Omega) = \partial_N\Omega$.
\end{itemize}

We now introduce the function spaces that will be used in the formulation of the $n$-th viscoelastic problem. For every $n \in \mathbb{N}$ and $t\in [0,T]$ let $V^n$, $V^n_t$, and $V^{n,D}_t$ be defined as $V$, $V_t$, and $V^{D}_t$ (see \eqref{eq:VeH} and \eqref{eq:def-spazi}) with $\Gamma$ and $\Gamma_t$ replaced by $\Gamma^n$ and $\Gamma^n_t$. Let $\mathcal{V}^n$, $\mathcal{V}^{n,D}$, and $\mathcal{V}^{n,\infty}$ be defined as $\mathcal{V}$, $\mathcal{V}^{D}$, and $\mathcal{V}^{\infty}$ (see \eqref{eq:Vcorsivo}, \eqref{eq:VcorsivoD}, and \eqref{eq:defVinfty}) with $V_t$ and $V^D_t$ replaced by $V^n_t$ and $V^{n,D}_t$.

For every $n\in\mathbb{N}$ we fix 
\begin{eqnarray}\label{eq:dati-pn}
     & u^{0,n}\in V^n_0, \quad u^{1,n}\in H, \quad u^n_D\in H^2(0,T;\,H) \cap H^1(0,T;\,V^n_0),  \\
     & f^n\in \Ls(0,T;H), \quad F^n\in H^1(0,T;\tilde{H}),
\end{eqnarray}
and we suppose that $u^{0,n}$ and $u^n_D$ satisfy the compatibility condition
\begin{equation}\label{eq:comp-conditt-pn}
    u^{0,n}-u^n_D(0)\in V^{n,D}_0.
\end{equation}

Now we give the detailed regularity and convergence hypotheses on the data. First of all we assume that there exists a constant $K>0$ such that for every $n\in \mathbb{N}$ the following Korn inequality is satisfied:
\begin{equation}\label{eq:korn-unif}
\|Dv\|^2 \leq K ( \|v\|^2 + \|Ev\|^2 ) \quad \textnormal{ for every $v \in H^1(\Omega\setminus\Gamma^n;\Rd)$.}
\end{equation}
We set $\underline{H}= L^2(\Omega, \Rdd)$.
Concernig the convergence of our data we assume that
\begin{equation}\label{eq:conv-diff}
    \| \Phi^n - \Phi \|_{C^2} \to 0, \qquad \| \Psi^n - \Psi \|_{C^2} \to 0,
\end{equation}
\begin{equation}\label{eq:conv-Vn}
\| \C^n - \C \|_{C^1} \to 0, \qquad
   \| \V^n - \V \|_{C^1} \to 0
\end{equation}
\begin{equation}\label{eq.convDir}
   \| u^n_D - u_D \|_{H^2(0,T;\,H)} \to 0, \qquad  \| Du^n_D - Du_D \|_{H^1(0,T;\underline{H})} \to 0 ,
\end{equation}
\begin{equation}\label{eq:conf-forz}
    \|f^n - f\|_{\Ls(0,T;H)}\to 0,
\qquad
    \|F^n - F\|_{ H^1(0,T;\tilde{H})}\to 0,
\end{equation}
\begin{equation}\label{eq.convu0u1}
    \| u^{0,n} - u^0 \| \to 0, \qquad \| Du^{0,n} - Du^0 \| \to 0, \qquad
    \| u^{1,n} - u^1 \| \to 0,
\end{equation}
\begin{equation}\label{hpboooo}
   \|\Theta^n- Id\|_{C^2}\to 0, \qquad \|\Xi^n- Id\|_{C^2}\to 0.
\end{equation}
It follows from (H19)-(H21) and \eqref{hpboooo} that
\begin{equation}\label{eq:bound-det}
   m_{det}(\Psi^n)\to m_{det}(\Psi)  \quad \text{and} \quad
   M_{det}(\Psi^n) \to M_{det}(\Psi^n) \quad \text{as } n\to \infty.
\end{equation}

For every $n\in\mathbb{N}$ we consider the solution $u^n$ of the problem
\begin{align}
    &\label{eq:weak_1-n} \,\,\,u^n\in\mathcal{V}^n \quad \text{and} \quad u^n-u^n_D\in \mathcal{V}^{n,D},\\
    &-\int^T_0(\dot u^n(t),\dot \varphi(t))\,\dt+ \int^T_0 ((\C^n+ \V^n) Eu^n(t),E\varphi(t)) \,\dt \vspace{1cm}\nonumber \\
    & - \int^T_0\int^t_0  \textnormal{e}^{\tau-t} (\V^n Eu^n(\tau), E\varphi(t))\,\dtau\dt= \int^T_0 (f^n(t),\varphi(t))\,\dt{} \nonumber \\
    &\label{eq:weak_2-n}+ \int^T_0  (F^n(t),E\varphi(t))  \,\dt{}  \quad\text{for all }  \varphi\in \mathcal{V}^{n,D} \text{ with } \varphi(0)=\varphi(T)=0,\\
    & \label{eq:weak_3-n}\,\,\,u^n(0)=u^{0,n} \quad \text{in }  H \quad  \text{and} \quad \dot u^n(0)=u^{1,n} \quad \text{in } (V^{D,n}_0)^*.
\end{align}
We also consider the solution $v^n$ of the problem
\begin{align}
    &\label{eq:onde_1-n} \,\,\,v^n\in\mathcal{V}^n \quad \text{and} \quad v^n-u^n_D\in \mathcal{V}^{n,D},\\
    &-\int^T_0(\dot v^n(t),\dot \varphi(t))\,\dt+ \int^T_0 (\A^n Ev^n(t),E\varphi(t)) \,\dt = \int^T_0 (f^n(t),\varphi(t))\,\dt{} \vspace{1cm}\nonumber \\
    &\label{eq:onde_2-n} + \int^T_0  (F^n(t),E\varphi(t))  \,\dt{}  \quad\text{for all }  \varphi\in \mathcal{V}^{n,D} \text{ with } \varphi(0)=\varphi(T)=0,\\
    & \label{eq:onde_3-n}\,\,\,v^n(0)=u^{0,n} \quad \text{in }  H \quad  \text{and} \quad \dot v^n(0)=u^{1,n} \quad \text{in } (V^{D,n}_0)^*.
\end{align}
The notion of convergence for $u^n$ as $n\to \infty$ can't be given directly because they don't belong to the same space. To overcome this problem we need to embed $V^n$ into a common space. This will be done using the standard embedding $V^n \hookrightarrow H \times \underline{H}$ given by $v  \mapsto (v,\,Dv) $, where the distrubutional gradient $Dv$ on $\Omega\setminus \Gamma^n$ is regarded as a function defined a.e. on $\Omega$, which belongs to $\underline{H}$.

We are now in a position to state one the main result of this section.
\begin{theorem}\label{thm:dip-cont-main}
Assume (H1)-(H23), \eqref{eq:forzante-f}-\eqref{eq:comp-conditt}, and \eqref{eq:dati-pn}-\eqref{hpboooo}. Let $u$ be the solution of \eqref{eq:weak_generalizzata1}-\eqref{eq:weak_generalizzata3} and let (for every $n\in \mathbb{N}$) $u^n$ be the solution of \eqref{eq:weak_1-n}-\eqref{eq:weak_3-n}. Then
\begin{equation*}
    (   u^n(t),\, D u^n(t), \, \dot{ u}^n(t) ) \to (  u(t),\, D u(t), \, \dot{ u}(t) ) \quad \text{ in } H\times \underline{H} \times H
 \end{equation*}
 for every $t \in [0,T]$. Moreover there exists a constant $C>0$ such that
\begin{equation*}
    \| u^n(t) \| + \| D u^n(t)\| + \| \dot{ u}^n(t) \| \leq C
\end{equation*}
for every $n\in \mathbb{N}$ and $t \in [0,T]$.
\end{theorem}

The proof is based on the following lemma.

\begin{lemma}\label{lemma:converg-punti-fissi}
 Let $X$ a complete metric space, let $G_n,\,G\colon X \to X$ with $n\in\mathbb{N}$ be maps with same contraction constant $\lambda \in (0,1)$, and let $x_n,\,x$ be the corresponding fixed points. Suppose that $G_n(y) \to G(y)$ for every $y\in X$. Then $x_n\to x$.
\end{lemma}
\begin{proof}
We have
$d(x_n,x)=$ $d(G_n(x_n),G(x))\leq d(G_n(x_n),G_n(x)) + d(G_n(x),G(x)) \break \leq \lambda d(x_n,x) + d(G_n(x),G(x)) $, hence
$(1 - \lambda)d(x_n,x) \leq  d(G_n(x),G(x))  \to 0 $, as $n\to + \infty$.
\end{proof}
In order to apply the previous lemma we will identify $u_n$ and $u$ with the fixed points of suitable operators defined in the Banach space
\begin{equation}\label{eq:Wpspace}
    \mathcal{W}:=L^2((0,T);H\times \underline{H} \times H),
\end{equation}
where on $H\times \underline{H} \times H$ we consider the Hilbert product norm defined by
\begin{equation}
    \|(h_1,h_2,h_3)\|_{H\times \underline{H} \times H}:= \Big( \|h_1\|^2 + \|h_2\|^2 +\|h_3\|^2\Big)^{1/2}
\end{equation}
for every $(h_1,h_2,h_3) \in H\times \underline{H} \times H$. In order to define the sequence of maps whose fixed points are $(u^n,Du^n,\dot u^n)$ and $(u,Du,\dot u)$, we consider the linear operators
\begin{equation}
    \mathcal{T}^n:\mathcal{W} \longrightarrow H^1(0,T;\tilde{H}) \quad \text{ and } \quad
    \mathcal{T}:\mathcal{W} \longrightarrow H^1(0,T;\tilde{H})
\end{equation}
defined as
\begin{equation}
    (\mathcal{T}^nw)(t):=\int^t_0  \textnormal{e}^{\tau-t} \V^n_e w_2(\tau) \,\dtau \quad \text{ and } \quad
    (\mathcal{T}w)(t):=\int^t_0  \textnormal{e}^{\tau-t} \V_e w_2(\tau) \,\dtau,
\end{equation}
where $w(t)=(w_1(t),w_2(t),w_3(t))$ and $\V^n_e,\,\V_e$ are as in \eqref{eq:tensori-estesi} and \eqref{eq:tensori-estesi-limite}. Arguing as in \eqref{eq:Lu-inf} and \eqref{eq:dotLu-inf} we get that
 \begin{equation}\label{eq:Tau-new}
    \|\mathcal{T}w\|_{L^\infty(0,T;\tilde{H})} \leq T^{{1}/{2}} \|\V\|_{\infty} \| w \|_{\mathcal{W}},
    \end{equation}
    \begin{equation}\label{eq:dotTau-new}
     \|\dot{\wideparen{\mathcal{T}w}}\|_{L^2(0,T;\tilde{H})}\leq (1+T)\|\V\|_{\infty} \| w \|_{\mathcal{W}},
\end{equation}
and the same estimate holds for $\mathcal{T}^nw$ with $\V$ replaced by $\V^n$.

Let $\mathcal{G}: \mathcal{W} \to \mathcal{W}$ be the operator defined for every $w\in \mathcal{W}$ by
\begin{equation}
    \mathcal{G}(w)=(z,Dz,\dot z),
\end{equation}
where $z$ is the solution of problem \eqref{eq:onde1}-\eqref{eq:onde3} with $F$ replaced by $F+ \mathcal{T}w$. From the definition of $\mathcal{G}$ it follows that $(u,Du,\dot u)$ is a fixed point of map $\mathcal{G}$ if and only if $u$ is the solution of the problem considered in Theorem \ref{thm:dip-cont-main}.

Similarly, let $\mathcal{G}^n: \mathcal{W} \to \mathcal{W}$ be the operator defined for every $w\in \mathcal{W}$ by
\begin{equation}
    \mathcal{G}^n(w)=(z^n,Dz^n,\dot z^n),
\end{equation}
where $z^n$ is the solution of problem \eqref{eq:onde_1-n}-\eqref{eq:onde_3-n} with $F$ replaced by $F^n+\mathcal{T}^nw$. From the definition of $\mathcal{G}^n$ it follows that $u^n$ is the solution of problem \eqref{eq:weak_1-n}-\eqref{eq:weak_3-n} if and only if $(u^n,Du^n,\dot u^n)$ is a fixed point of map $\mathcal{G}^n$.

The following lemma provides a uniform Lipschitz estimate for the operators $\mathcal{G}^n$.

\begin{proposition}\label{lemma:equicontratt} There exist a positive constants $B$, independent of $n$ and $T$, such that
\begin{equation}\label{eq:equicontrGn-new}
        \|\mathcal{G}^n(w_1)-\mathcal{G}^n(w_2)\|_{\mathcal{W}}\leq B (T + T^3) \| w_1-w_2 \|_{\mathcal{W}},
\end{equation}
for every $w_1,\,w_2\in \mathcal{W}.$
\end{proposition}
\begin{proof}
Let us fix $w_1,\,w_2\in \mathcal{W}$ and set $w:=w_1-w_2$. We observe that $\mathcal{G}^n(w_1)-\mathcal{G}^n(w_2)=(z^n,Dz^n,\dot z^n)$ where $z^n$ is the solution of problem \eqref{eq:weak_1-n}-\eqref{eq:weak_3-n} with $F^n$ replaced by $\mathcal{T}^nw$ and $u^n_D$, $f^n$, $u^{0,n}$, $u^{1.n}$ replaced by zero. From Theorem \ref{prop:energ-soloF} and from the uniform bound of the data there exists a positive constants $A$, independent of $n$ and $T$, such that
\begin{equation}
        \|z^n\|_{\mathcal{V}^\infty}\leq A(1+T)\|\mathcal{T}^nw \|_{L^\infty(0,T;\tilde{H})}+ A(T^{1/2}+ T^{3/2})\|{\dot{\wideparen{\vphantom{\mathcal{T}}\smash{\mathcal{T}^nw}}}}\|_{L^2(0,T;\tilde{H})}.
    \end{equation}
    Using \eqref{eq:Tau-new} and \eqref{eq:dotTau-new} we get
    \begin{equation}
        \|(z^n,Dz^n,\dot z^n)\|_{\mathcal{W}}\leq A\Big((1+T) T + (T^{1/2}+T^{3/2})^2\Big)\|\V^n\|_{\infty} \| w \|_{\mathcal{W}}
    \end{equation}
    which gives \eqref{eq:equicontrGn-new} taking into account \eqref{eq:conv-Vn}.
\end{proof}

To apply Lemma \ref{lemma:converg-punti-fissi} we have to prove that
\begin{equation*}
\mathcal{G}^n(w) \to \mathcal{G}(w) \quad \textnormal{in } \mathcal{W},
\end{equation*}
for every $w\in \mathcal{W}$. In order to prove this we will use the results for the wave equation developed in \cite{Caponi-tesi}. 
Unfortunately these results can not be applied directly because they are obtained under the assumptions:
\begin{itemize}
    \item[(a)] $\Gamma^n_0=\Gamma_0$ for all $n\in\mathbb{N}$,
    \item[(b)] the forcing terms belong to $L^2(0,T; H)$.
\end{itemize}
To overcome the difficulties due to (a) we need some preliminary results. The first one is an uniform bound of the solution of problems \eqref{eq:onde_1-n}-\eqref{eq:onde_3-n}.

\begin{proposition}\label{prop:equibdd-vn}
Assume (H1)-(H23), \eqref{eq:dati-pn}-\eqref{eq.convDir}, and \eqref{eq.convu0u1}-\eqref{hpboooo}. Let let $v^n$ be the solution of \eqref{eq:onde_1-n}-\eqref{eq:onde_3-n}. Then the there exists a positive constant $C$ such that
\begin{equation}\label{eq:equibdd-vn}
    \|v^n\|_{\mathcal{V}^{n,\infty}} \leq C \quad \text{for every } n\in \mathbb{N}.
\end{equation}
\end{proposition}
\begin{proof}
We note that $v^n_0(t):=v^n(t)-u^{0,n}+u^n_D(0)-u^n_D(t)$ is the solution of \eqref{eq:onde_1-n}-\eqref{eq:onde_3-n} with $u^0$ replaced by $0$, $u^{1,n}$ replaced by $u^{1,n}-\dot u^n_D(0)$, $u^n_D$ replaced by $0$, $f^n$ replaced by $f^n-\ddot u^n_D$, and $F^n$ replaced by $F^n-\A^nEu^n_D-\A^nE(u^{n,0}-u^n_D(0))$. Then we can apply Proposition \ref{prop:energ-soloF} and \eqref{eq:conv-diff}-\eqref{hpboooo} to obtain that $\|v^n_0\|_{\mathcal{V}^n}$ is equibounded. By \eqref{eq.convDir} and \eqref{eq.convu0u1} we get \eqref{eq:equibdd-vn}.
\end{proof}

The next proposition deals with the case of solution of \eqref{eq:onde_1-n}-\eqref{eq:onde_3-n} when $F^n$ is replaced by $0$. 

\begin{proposition}\label{lemma:convergenza(g,0)}
Assume (H1)-(H23), \eqref{eq:dati-pn}-\eqref{eq.convDir}, and \eqref{eq.convu0u1}-\eqref{hpboooo}. Given $g\in L^2(0,T;\,H)$, let $v^n$ be the solution of \eqref{eq:onde_1-n}-\eqref{eq:onde_3-n} with $f^n$ replaced by $g$ and $F^n$ replaced by $0$. Let $v$ be the solution in \eqref{eq:onde1}-\eqref{eq:onde3} with $f$ replaced by $g$ and $F$ replaced by $0$.
Then for every $t \in [0,T]$ we have 
\begin{equation}\label{lemma:convergenza(g,0)-p}
    (   v^n(t),\, D v^n(t), \, \dot{ v}^n(t) ) \to (  v(t),\, D v(t), \, \dot{ v}(t) ) \quad \text{ in } H\times \underline{H} \times H.
 \end{equation}
\end{proposition}

In order to prove this proposition it is convenient to use the following elementary result, whose proof, based on a change of variables, is omitted (for a similar result see \cite[Lemma A.7]{DalMaso-Luc}).

\begin{lemma}\label{lemma:conv-composiz}
For every $n \in \mathbb{N}$ let $h^n, \, h \in H$ and let $\Lambda^n,\,\Lambda: \overline{\Omega} \to \overline{\Omega}$ be $C^1$ diffeomorphisms. Assume that $h^n \to h$ in $H$ and $\Lambda^n \to \Lambda$ in $C^1$. Assume also that $ \det D\Lambda^n(x) >0$ and $ \det D\Lambda(x) >0$ for every $x\in \overline\Omega$ and $n \in \mathbb{N}$. Then $h^n\circ \Lambda^n \to h\circ \Lambda$ as $n \to \infty$ in $H$.
\end{lemma}

\begin{proof}[Proof of Proposition \ref{lemma:convergenza(g,0)}]
To overcome the difficulty due to the fact that we may have $\Gamma_0^n\neq \Gamma_0$, by a change of variables we transform our problem into a problem with new cracks $\hat\Gamma^n_t$ satisfying $\hat\Gamma^n_0=\Gamma_0$ for every $n$, to which we can apply the results of \cite{Caponi} and \cite{Caponi-tesi}.

For every $n$ and $t$ we define $\hat\Gamma^n_t:=\Xi^n(\Gamma^n_t)\subset\Gamma$ and observe that $\hat\Gamma^n_t$ satisfies (H10). The vector spaces $\hat{V}^n_t$ and $\hat{V}^{n,D}_t$ are defined as ${V}^n_t$ and ${V}^{n,D}_t$ (see \eqref{eq:def-spazi}) with $\Gamma_t$ replaced by $\hat\Gamma^n_t$, while $\hat{\mathcal{V}}^{n}$ and $\hat{\mathcal{V}}^{n,D}$ are defined as ${\mathcal{V}}^{n}$ and ${\mathcal{V}}^{n,D}$ (see \eqref{eq:Vcorsivo} and \eqref{eq:VcorsivoD}) with $V_t$ and $V^D_t$ replaced by $\hat{V}^n_t$ and $\hat{V}^{n,D}_t$.

For every $t\in [0,T]$ let $\hat v^n(t):=v^n(t)\circ \Theta^n $, $\hat u^n_D(t):= u^n_D(t) \circ \Theta^n$, $\hat u^{0,n}:=u^{0,n}\circ \Theta^n $, $\hat u^{1,n}:=u^{1,n}\circ\Theta^n$, and $\hat{g}^n(t):=g(t) \circ \Theta^n$. It is easy to see that $\hat v^n\in \hat{\mathcal{V}}^{n}$, $\hat v^n-\hat u^n_D\in \hat{\mathcal{V}}^{n,D}$, $\hat v^n(0)=\hat u^{0,n},\,\dot{\hat v}^n(0)=\hat u^{1,n}$.

To write the equation satisfied by $\hat v^n$ we introduce $\hat{\mathbb{A}}^n\colon \overline{\Omega} \to \mathcal{L}(\Rdd;\Rdd)$ defined as
\begin{equation}\label{eq:hatAn}
        \hat{\mathbb{A}}^n(y)[A]:={\A}^n_e(\Theta^n(y)){\big[} AD\Xi^n(\Theta^n(y)){\big]}(D\Xi^n(\Theta^n(y)))^T \quad \text{ for all } A \in \Rdd,
\end{equation}
where $\A^n$ is defined in \eqref{eq:A-An}. We note that $\hat{\A}^n$ is of class $C^1$, with equibounded $C^1$ norm. Moreover it is symmetric on $\mathcal{L}(\Rdd,\Rdd)$.

Setting $h^n(x):=\nabla[\det D\Xi^n(x)  ]$, we introduce $\mathbb{L}^n\colon \overline{\Omega} \to \mathcal{L}(\Rdd;\Rd)$ defined as
\begin{equation*}
    \mathbb{L}^n(y)[A]={\A}^n_e(\Theta^n(y)){\big[} A D\Xi^n(\Theta^n(y)){\big]}  h^n(\Theta^n(y)) \,\det D\Theta^n(y) \quad \text{ for all } A \in \Rdd.
\end{equation*}
Let $\varphi \in \hat{\mathcal{V}}^{n,D} $ with $\varphi(0)=\varphi(T)=0$. Using $(\varphi(t)\circ \Xi^n)\det D\Xi^n$ as test function in the equation for $v^n(t)$ we get
\begin{equation*}
   - \int^T_0 ( \dot{\hat v}^n(t),\dot{\varphi}(t))\,\dt + \int^T_0  (\hat{\mathbb{A}}^nD\hat v^n(t), D\varphi(t))\,\dt + \int^T_0 (\mathbb{L}^nD\hat v^n(t) ,  \varphi(t)) \, \dt =
\end{equation*}
\begin{equation*}
    \int^T_0 (\hat{g}^n(t) , \varphi(t))\,\dt.
\end{equation*}
By Proposition \eqref{prop:equibdd-vn} the sequence $||v^n||_{\mathcal{V}^n}$ is bounded and in particular $\|Dv^n(t)\|$ is uniformly bounded with respect to $n$ and $t$. By the definition of $\hat v^n$ and \eqref{hpboooo} also $\|D\hat v^n(t)\|$ is uniformly bounded with respect to $n$ and $t$. Since $\det D\Xi^n \to 1$ in $C^1(\overline{\Omega})$, we have $\nabla[\det D\Xi^n] \to 0$ in $C^0(\overline{\Omega},\,\Rd{})$, which implies that $\mathbb{L}^n \to 0$ uniformly as $n\to +\infty$. From this fact and the uniform bound on $\|D \hat v^n(t)\|$ we get
\begin{equation}\label{eq:hatLn}
    \|\mathbb{L}^nD\hat v^n(t)\| \to 0 \text{ as } n\to +\infty,
\end{equation}
uniformly in $t$. Therefore, setting
\begin{equation}\label{eq:hatfn}
    {\hat f}^n:= \hat{g}^n - \mathbb{L}^nD\hat v^n,
\end{equation}
we conclude that
\begin{align}
    &\label{eq:hat1} \,\,\,\hat v^n\in \hat{\mathcal{V}}^{n} \quad \text{and} \quad \hat v^n-\hat u^n_D\in \hat{\mathcal{V}}^{n,D},\\
    &- \int^T_0 ( \dot{\hat{v}}^n(t),\dot{{\varphi}}(t))\,\dt + \int^T_0  (\hat{\mathbb{A}}^nD\hat v^n(t), D\varphi(t))\,\dt   = \int^T_0 (\hat f^n(t) , \varphi(t))\,\dt,\nonumber\\
    &\label{eq:hat2} \text{ for all } \varphi\in \hat{\mathcal{V}}^{n,D} \text{ such that } \varphi(0)=\varphi(T)=0,\\
    & \label{eq:hat3}\,\,\,\hat v^n(0)=\hat u^{0,n} \quad \text{in }  H \quad  \text{and} \quad  \dot{\hat v}^n(0)=\hat u^{1,n} \quad \text{in } (V^D_0)^*.
\end{align}
In order to apply the results of \cite{Caponi-tesi} we define $\hat\Phi^n(t, y):=\Xi^n( \Phi^n(t, \Theta^n(y)))$, $\hat\Psi^n(t, x):=\Xi^n( \Psi^n(t, \Theta^n(x)))$. We observe that $\hat\Phi^n$ and $\hat\Psi^n$ satisfy (H11)-(H14) with $\Gamma_t$ replaced by $\hat\Gamma^n_t$.
Since in general $\hat{\A}^n[A]\neq \hat{\A}^n[A_{sym}] $ for some $A\in \Rdd$, we cannot apply the results of \cite{Caponi}. However it is possible to use the results of \cite{Caponi-tesi} which hold under more general assumptions involving the tensor
\begin{equation*}
\hat{\mathbb{B}}^n(t, y)[A]:= \hat{\mathbb{A}}^n(\hat\Phi^n(t,y))[A D\hat\Psi^n(t, \hat\Phi^n(t,y))] D\hat\Psi^n(t, \hat\Phi^n(t, y))^T 
\end{equation*}
\begin{equation*}
    - A {\dot{\hat\Psi}^n(t, \hat\Phi^n(t,y))}
    \otimes
    {\dot{\hat\Psi}^n(t, \hat\Phi^n(t,y))},
\end{equation*}
for all $A \in \Rdd{}$, $t\in [0,T]$, $y\in \overline{\Omega}$. We claim that there exists two constants $c_0,\,c_1>0$ (independent of $n$) such that, for $n$ large enough, we have
\begin{equation}\label{eq:bound-su-hatBn}
   (\hat{\mathbb{B}}^n(t)D\varphi,\,D\varphi) \geq c_0 \|\varphi\|^2_{V_0} - c_1\| \varphi\|^2
\end{equation}
for all $\varphi \in V_0 $ and $t \in [0,T]$. This is the hypothesis on $\hat{\mathbb{B}}^n$ required in \cite{Caponi-tesi}. 

To prove the claim we use (H3), (H15), and \eqref{hpboooo} (which are satisfied uniformly in $n$) and by standard computations (see, for instance, \cite[Section 1.2]{Caponi-tesi}) we obtain
\begin{align}
   (\hat{\mathbb{B}}^n(t)D\varphi,\,D\varphi) \geq& \int_\Omega \big|D\varphi(y)D\Xi^n(\Theta^n(y))D\Psi^n(t,\Phi^n(t,\Theta^n(y)))|^2 \omega^n(t,y) \,dy\nonumber\\
   &-\alpha_0\min_{[0,T] \times\overline{\Omega}}\{\det D\Xi^n\det D\Psi^n\} \int_\Omega |\varphi(\Xi^n(\Psi^n(t,y)))|^2\,dy\label{eq:contacciBn}
\end{align}
where
\begin{equation*}
    \omega^n(t,y):=\frac{\alpha_0 m_{det}(\Psi^n)}{KM_{det}(\Psi^n)}\min_{\overline{\Omega}}\{\det D\Xi^n\}\min_{\overline{\Omega}}\{\det D\Theta^n\}-|\dot \Phi^n(t,\Theta^n(y))|^2,
\end{equation*}
while $m_{det}(\Psi^n)$, $M_{det}(\Psi^n)$, $\alpha_0$, and $K$ are the constants that appear in (H15), (H16), and \eqref{eq:korn-unif}.
Since the inverse of the matrices  $D\Xi^n(x)D\Psi^n(t,\Phi^n(t,x))$ are bounded uniformly with respect to $n$, $t$, and $x$, there exists a constant $\beta>0$ such that 
\begin{equation*}
    \int_\Omega \big|D\varphi(y)D\Xi^n(\Theta^n(y))D\Psi^n(t,\Phi^n(t,\Theta^n(y)))|^2 \omega^n(t,y) \,dy \geq \beta \int_\Omega \big|D\varphi(y)|^2 \omega^n(t,y) \,dy
\end{equation*}
for all $n$ and $t$. Moreover by \eqref{eq:conv-diff} and \eqref{hpboooo} there exists a constant $\gamma>0$ such that 
\begin{equation*}
    \alpha_0\min_{[0,T] \times\overline{\Omega}}\{\det D\Xi^n\det D\Psi^n\} \int_\Omega |\varphi(\Xi^n(\Psi^n(t,y)))|^2\,dy \leq \gamma  \int_\Omega |\varphi(y)|^2\,dy
\end{equation*}
for all $n$ and $t$. Therefore \eqref{eq:contacciBn} gives
\begin{equation}\label{eq:quasiboundsuBn}
    (\hat{\mathbb{B}}^n(t)D\varphi,\,D\varphi) \geq \beta \int_\Omega \big|D\varphi(y)|^2 \omega^n(t,y) \,dy - \gamma  \int_\Omega |\varphi(y)|^2\,dy.
\end{equation}
To conclude the proof of the claim, we define
\begin{equation*}
    \omega(t,y):=\frac{\alpha_0 m_{det}(\Psi)}{KM_{det}(\Psi)}-|\dot \Phi(t,y)|^2.
\end{equation*}
By \eqref{eq:conv-diff}, \eqref{hpboooo} and \eqref{eq:bound-det}, we have $\omega^n \to \omega$ uniformly on $[0,T]\times \overline\Omega$. By (H15) and by continuity there exists $\varepsilon>0$ such that $\omega(t,y)\geq 2\varepsilon$ for all $(t,y)\in [0,T]\times \overline\Omega$. By uniform convergence there exists $n_\varepsilon$ such that $\omega^n(t,y)\geq \varepsilon$ for all $(t,y)\in [0,T]\times \overline\Omega$ and for all $n>n_\varepsilon$. This inequality together with \eqref{eq:quasiboundsuBn} implies \eqref{eq:bound-su-hatBn} and concludes the proof of the claim.

By \eqref{eq:conv-diff} and \eqref{hpboooo} we get $\hat\Phi^n \to \Phi$ and $\hat\Psi^n \to \Psi$ in $C^2$, while \eqref{eq:hatAn} and \eqref{hpboooo} give $\hat \A^n \to \A$ in $C^1$. Moreover applying Lemma \ref{lemma:conv-composiz} to the functions and their derivatives we can prove that $\hat u^{0,n} \to u^0 $ in $V_0$, $\hat u^{1,n} \to u^1$ in $H$, $\hat u^n_D \to u_D$ in $H^2(0,T; H) \cap H^1(0,T; V_0)$, and $\hat g^n \to g$ in $L^2(0,T; H)$. Using \eqref{eq:hatLn} and \eqref{eq:hatfn} we have that $\hat f^n \to g$ in $L^2(0,T; H)$. We are now in a position to apply \cite[Theorem 1.4.1]{Caponi-tesi} to problem \eqref{eq:hat1}-\eqref{eq:hat3} and we obtain
\begin{equation*}
    (  \hat v^n(t),\, D\hat v^n(t), \, \dot{\hat v}^n(t) ) \to (  v(t),\, D v(t), \, \dot{ v}(t) ) \quad \text{ in } H\times \underline{H} \times H
 \end{equation*}
 for every $t \in [0,T]$. Since $$v^n(t,\cdot) = \hat v^n (t, \Xi^n(\cdot)),\quad Dv^n(t,\cdot) = D \hat v^n (t, \Xi^n(\cdot))D\Xi^n(\cdot), \quad \dot{v}^n(t,\cdot) = \dot{\hat v}^n (t, \Xi^n(\cdot)),$$ using Lemma \ref{lemma:conv-composiz} we get \eqref{lemma:convergenza(g,0)-p}
 for every $t \in [0,T]$.
\end{proof}

To use Proposition \ref{lemma:convergenza(g,0)} in the proof of the convergence $\mathcal{G}^n(w) \to \mathcal{G}(w)$ we need the following approximation result.

\begin{lemma}\label{lemma:approx-forzante}
 Let $G\in H^1((0,T);\,\tilde{H})$. For every $\varepsilon > 0$ there exists a compact neighborhood $K_\varepsilon$ of $\Gamma \cap \overline{\Omega}$ and $G_\varepsilon\in H^1((0,T);\, \tilde{H})$ such that $G_\varepsilon(t)\in C^\infty_c(\Omega\setminus K_\varepsilon;\,\mathbb{R}^{d\times d}_{sym})$ for every $t\in [0,T]$ and $$\|G_\varepsilon - G\|_{L^\infty(0,T;\tilde{H})} + \|\dot G_\varepsilon - \dot G\|_{L^2(0,T;\tilde{H})} < \varepsilon.$$
\end{lemma}

\begin{remark} By (H22) and \eqref{hpboooo} for every $\varepsilon>0$ there exists $n_\varepsilon$ such that $\Gamma^n \subset K_\varepsilon$, for $n>n_\varepsilon$. From the properties of $G_\varepsilon$ follows that
 \begin{equation}\label{eq:int-parti}
     (G_\varepsilon(t), \, Ev)=-(\dive G_\varepsilon(t), \, v)
 \end{equation}
 for all $t\in [0,T]$ and for all $v\in V_n$, for $n>n_\varepsilon$.
 \end{remark}

\begin{proof}[Proof of Lemma \ref{lemma:approx-forzante}]
Given a partition of $[0,T]$, we can consider the piecewise affine interpolation of the values of $F$ at the nodes. It is well known that this interpolation converges in $H^1(0,T;\tilde{H})$ to $F$ as the fineness of the partition tends to zero. To conclude, it is enough to approximate in $\tilde{H}$ the values of $F$ at the nodes by elements of $C^\infty_c(\Omega\setminus\Gamma;\,\mathbb{R}^{d\times d}_{sym})$ and to consider the corresponding piecewise affine interpolation.
\end{proof}

\begin{proposition}\label{cor:convergenza(fn,Gn),Gamma0n-NEW}
Assume (H1)-(H23) and \eqref{eq:korn-unif}-\eqref{hpboooo}. Let $v^n$ be the solution of \eqref{eq:onde_1-n}-\eqref{eq:onde_3-n} and let $v$ be the solution of \eqref{eq:onde1}-\eqref{eq:onde3}. Then for every $t \in [0,T]$ we have
\begin{equation}\label{conv1-new}
    (   v^n(t),\, D v^n(t), \, \dot{ v}^n(t) ) \to (  v(t),\, D v(t), \, \dot{ v}(t) ) \quad \text{ in } H\times \underline{H} \times H.
 \end{equation}
 Moreover
 \begin{equation}\label{conv2-new}
    (v^n,\,Dv^n,\,\dot{v}^n) \to (v,\,Dv,\,\dot{v}) \quad \textnormal{in }\mathcal{W}=L^2((0,T); H\times \underline{H} \times H).
\end{equation}
\end{proposition}

\begin{proof}
Let $\varepsilon >0 $, let $G_\varepsilon$ the function in Lemma \ref{lemma:approx-forzante} with $G=F$. Let $v^n_\varepsilon$ solution of \eqref{eq:onde_1-n}-\eqref{eq:onde_3-n} with $f^n$ and $F^n$ replaced by $f$ and $G_\varepsilon$, let $v^\varepsilon$ solution of \eqref{eq:onde1}-\eqref{eq:onde3} with $F$ replaced by $G_\varepsilon$. By \eqref{eq:conf-forz} there exists $n_\varepsilon$ such that 
\begin{equation}\label{eq:dis1-NEW}
    \|f^n-f\|_{L^2(0,T;H)} + \|F^n - G_\varepsilon\|_{L^\infty(0,T;\tilde{H})}  + \|\dot F^n - \dot G_\varepsilon\|_{L^2(0,T;\tilde{H})}  < \varepsilon
\end{equation}
for every $n>n_\varepsilon.$ The function $v^n-v^n_\varepsilon$ is the solution of problem \eqref{eq:onde_1-n}-\eqref{eq:onde_3-n} with $f^n$ and $F^n$ replaced by $f^n-f$ and $F^n-G_\varepsilon$ and $u^n_D$, $f^n$, $u^{n,0}$, $u^{1.n}$ replaced by zero. Then by Proposition \ref{prop:energ-soloF} there exists a constant $C(T)$ depending on $T$ (independent of $n$ and $\varepsilon$) such that
\begin{equation}\label{eq:eps-new-n}
    \|v^n - v^n_\varepsilon\|_{\mathcal{V}^{n,\infty}} \leq C(T) \varepsilon .
\end{equation}
for every $ n > n_\varepsilon$. Similarly we  can prove
\begin{equation}\label{eq:eps-new}
    \|v - v_\varepsilon\|_{\mathcal{V}^{\infty}} \leq C(T) \varepsilon.
\end{equation}

 Changing the value of $n_\varepsilon$, by \eqref{eq:int-parti} we have that $v^n_\varepsilon$ is the solution of \eqref{eq:onde_1-n}-\eqref{eq:onde_3-n} with $f^n$ replaced by $g_\varepsilon:=f-\dive G_\varepsilon$ and $F^n$ replaced by $0$, while $v_\varepsilon$ is the solution of \eqref{eq:onde1}-\eqref{eq:onde3} with $f$ replaced by $g_\varepsilon:=f-\dive G_\varepsilon$ and $F$ replaced by $0$. By Proposition \ref{lemma:convergenza(g,0)} for every $t \in [0,T]$ we have 
\begin{equation}\label{eq:okasdbvoafb}
    (   v^n_\varepsilon(t),\, D v^n_\varepsilon(t), \, \dot{ v}^n_\varepsilon(t) ) \to (  v_\varepsilon(t),\, D v_\varepsilon(t), \, \dot{ v}_\varepsilon(t) ) \quad \text{ in } H\times \underline{H} \times H.
 \end{equation}
Since 
\begin{align*}
    & \|(   v^n(t),\, D v^n(t), \, \dot{ v}^n(t) ) - (  v(t),\, D v(t), \, \dot{ v}(t) )\| \leq \|v^n - v^n_\varepsilon\|_{\mathcal{V}^{n,\infty}} \\
    &  + \| (   v^n_\varepsilon(t),\, D v^n_\varepsilon(t), \, \dot{ v}^n_\varepsilon(t) ) - (  v_\varepsilon(t),\, D v_\varepsilon(t), \, \dot{ v}_\varepsilon(t) ) \| + \|v - v_\varepsilon\|_{\mathcal{V}^{\infty}},
\end{align*}
by \eqref{eq:eps-new-n}-\eqref{eq:okasdbvoafb} we get
 \begin{equation*}
     \limsup_{n\to + \infty}  \|(   v^n(t),\, D v^n(t), \, \dot{ v}^n(t) ) - (  v(t),\, D v(t), \, \dot{ v}(t) )\| \leq 2C(T)\varepsilon
 \end{equation*}
 for every $t\in [0,T]$. By the arbitrareness of $\varepsilon$ we obtain \eqref{conv1-new}. Finally, using the estimate in Proposition \ref{prop:equibdd-vn} and the Dominated Convergence Theorem we obtain \eqref{conv2-new}.
\end{proof}

\begin{corollary}\label{rem:convergenza(fn,Gn),Gamma0n-NEW}
 Assume (H1)-(H23) and \eqref{eq:korn-unif}-\eqref{hpboooo}. Then for every $w \in \mathcal{W}$ we have
 \begin{equation*}
     \mathcal{G}^n(w) \to \mathcal{G}(w) \quad \mathcal{W}.
 \end{equation*}
\end{corollary}
\begin{proof}
By \eqref{eq:conv-Vn} we get $\mathcal{T}^nw \to \mathcal{T}w$ in $H^1(0,T; \tilde{H})$  for every $w \in \mathcal{W}$. The result follows from Proposition \ref{cor:convergenza(fn,Gn),Gamma0n-NEW} with $F^n$ and $F$ replaced by $F^n + \mathcal{T}^nw$ and $F+ \mathcal{T}w$.
\end{proof}

As a consequence of Lemma \ref{lemma:converg-punti-fissi}, Proposition \ref{lemma:equicontratt}, and Corollary \ref{rem:convergenza(fn,Gn),Gamma0n-NEW} we obtain the continuous dependence result when $T$ is small enough.

\begin{theorem}\label{thm:cont-Tpiccolo-NEW}
Assume that $ B (T + T^3) <1$, where $B$ is the constant in Proposition \ref{lemma:equicontratt}. Then the conclusion of Theorem \ref{thm:dip-cont-main} holds.
\end{theorem}

\begin{proof}
By Corollary \ref{rem:convergenza(fn,Gn),Gamma0n-NEW} $\mathcal{G}^n(w) \to \mathcal{G}(w)$ in $\mathcal{W}$ for every $w \in \mathcal{W}$. By Proposition \ref{lemma:equicontratt} the maps $\mathcal{G}^n$ have the same contraction constant $ B (T + T^3) <1$. Then we are in a position to apply Lemma \ref{lemma:converg-punti-fissi} and we get
\begin{equation}\label{eq:wntow}
    w^n:=(u^n,\,Du^n,\,\dot{u}^n) \to (u,\,Du,\,\dot{u})=:w \quad \textnormal{in } \mathcal{W}=L^2((0,T); H\times \underline{H} \times H).
\end{equation}
From this convergence and \eqref{eq:conv-Vn}, we obtain $\mathcal{T}^nw^n \to \mathcal{T}w$ in $H^1(0,T; \tilde{H})$ and we can apply Proposition \ref{cor:convergenza(fn,Gn),Gamma0n-NEW}, with forcing term $F^n$ and $F$ replaced by $F^n + \mathcal{T}^nw^n $ and $F + \mathcal{T}w $. Since $F^n+\mathcal{T}^nw^n \to F+\mathcal{T}w$ in $H^1(0,T; \tilde{H})$ we get
\begin{equation*}
    (   u^n(t),\, D u^n(t), \, \dot{ u}^n(t) ) \to (  u(t),\, D u(t), \, \dot{ u}(t) ) \quad \text{ in } H\times \underline{H} \times H
 \end{equation*}
for every $t \in [0,T]$. We can apply Proposition \ref{prop:equibdd-vn} with $F^n$ replaced by $F^n+\mathcal{T}^nw^n$ and we obtain that there exists a constant $C>0$ such that
\begin{equation*}
    \| u^n(t) \| + \| D u^n(t)\| + \| \dot{ u}^n(t) \| \leq C
\end{equation*}
for every $n\in \mathbb{N}$ and $t \in [0,T]$.
\end{proof}

We are now in a position prove Theorem \ref{thm:dip-cont-main} without additional assumptions on $T$.

\begin{proof}[Proof of Theorem \ref{thm:dip-cont-main}]
There exists $k\in \mathbb{N}$ such that $T_0:=T/k$ satisfies  $ B (T_0 + T^3_0) <1$. By Theorem \ref{thm:cont-Tpiccolo-NEW} we have
\begin{equation}\label{eq:convT0punt}
    (   u^n(t),\, D u^n(t), \, \dot{ u}^n(t) ) \to (  u(t),\, D u(t), \, \dot{ u}(t) ) \quad \text{ in } H\times \underline{H} \times H\, \text{ for all }t \in [0,T_0],
 \end{equation}
 \begin{equation}\label{eq:convT0lp}
    (u^n,\,Du^n,\,\dot{u}^n) \to (u,\,Du,\,\dot{u}) \quad \textnormal{in }L^2((0,T_0); H\times \underline{H} \times H).
\end{equation}
If $k=1$ the proof is finished, otherwise we consider the problem on the interval $[T_0,2T_0]$.

Note that $u^n(T_0)\in V^n$ and $\dot u^n(T_0)\in H$ are well defined, because $u\in C^0_w([0,T_0];V^n)$ and $\dot u\in C^0_w([0,T_0];H)$. Since $u^n(t)\in V^n_t$ for a.e. $t\in (0,T_0)$, it easy to see that $u^n(T_0)\in V^n_{T_0}$. In order to study the problem on $[T_0,\,2T_0]$ we define the spaces $\mathcal{V}_{T_0,2T_0}$, $\mathcal{V}^{D}_{T_0,2T_0}$, $\mathcal{V}^{\infty}_{T_0,2T_0}$,
$\mathcal{V}^n_{T_0,2T_0}$, $\mathcal{V}^{n,D}_{T_0,2T_0}$,
$\mathcal{V}^{n,\infty}_{T_0,2T_0}$, and $\mathcal{W}_{T_0.2T_0}$ as $\mathcal{V}$, $\mathcal{V}^{D}$, $\mathcal{V}^{\infty}$,
$\mathcal{V}^n$, $\mathcal{V}^{n,D}$,
$\mathcal{V}^{n,\infty}$, and $\mathcal{W}$
 with $0$ and $T$ replaced by $T_0$ and $2T_0$. For every $t \in [T_0,2T_0]$ we set
\begin{equation*}
    G(t):= F(t) + \int^{T_0}_0  \textnormal{e}^{\tau-t} \V Eu(\tau) \dtau \quad \text{and} \quad
    G^n(t):= F^n(t) + \int^{T_0}_0  \textnormal{e}^{\tau-t} \V^n Eu^n(\tau) \dtau.
\end{equation*}
Let $v$ be the solution of the problem
\begin{align*}
    & \,\,\,v\in\mathcal{V}_{T_0,2T_0} \quad \text{and} \quad v-u_D\in \mathcal{V}^D_{T_0,2T_0},\\
    &-\int^{2T_0}_{T_0}(\dot v(t),\dot \varphi(t))\,\dt+ \int^{2T_0}_{T_0} (\A Ev(t),E\varphi(t)) \,\dt - \int^{2T_0}_{T_0}\int^t_{T_0} \textnormal{e}^{\tau-t} (\V Ev(\tau), E\varphi(t))\,\dtau\dt\vspace{1cm}\nonumber \\
    &  = \int^{2T_0}_{T_0} (f(t),\varphi(t))\,\dt{} + \int^{2T_0}_{T_0} (G(t),E\varphi(t))  \,\dt{}\text{ for every } \varphi\in \mathcal{V}^D_{T_0,2T_0} \text{with } \varphi({T_0})=\varphi({2T_0})=0, \\
    & \,\,\,v(T_0)=u(T_0) \quad \text{in }  H \quad  \text{and} \quad \dot  v(T_0)=\dot u(T_0) \quad \text{in } (V^D_{T_0})^*.
\end{align*}

For every $n\in \mathbb{N}$ let $v^n$ be the solution of the problem
\begin{align*}
    & \,\,\,v^n\in\mathcal{V}^n_{T_0,2T_0} \quad \text{and} \quad v^n-u^n_D\in \mathcal{V}^{n,D}_{T_0,2T_0},\\
    &-\int^{2T_0}_{T_0}(\dot v^n(t),\dot \varphi(t))\,\dt+ \int^{2T_0}_{T_0} (\A^n Ev^n(t),E\varphi(t)) \,\dt - \int^{2T_0}_{T_0}\int^t_{T_0} \textnormal{e}^{\tau-t} (\V^n Ev^n(\tau), E\varphi(t))\,\dtau\dt\vspace{1cm}\nonumber \\
    &  = \int^{2T_0}_{T_0} (f^n(t),\varphi(t))\,\dt{} + \int^{2T_0}_{T_0} (G^n(t),E\varphi(t))  \,\dt{}\text{ for every } \varphi\in \mathcal{V}^{n,D}_{T_0,2T_0} \text{with } \varphi({T_0})=\varphi({2T_0})=0, \\
    & \,\,\,v^n(T_0)=u^n(T_0) \quad \text{in }  H \quad  \text{and} \quad \dot  v^n(T_0)=\dot u^n(T_0) \quad \text{in } (V^{n,D}_{T_0})^*.
\end{align*}

We note that, by the definition of $G$ and $G^n$, the restrictions of $u$ and $u^n$ to $[T_0,2T_0]$ satisfy the problems for $v$ and $v^n$. By uniqueness we have that $v=u$ and $v^n=u^n$ on $[T_0,2T_0]$.

For every $x \in \overline{\Omega}$ and $[T_0,2T_0]$ we define
$\Phi_{T_0}(t,x):= \Phi(t,\Psi(T_0,x))$, $\Psi_{T_0}(t,x):= \Psi(t,\Phi(T_0,x))$ $\Phi^n_{T_0}(t,x):= \Phi^n(t,\Psi^n(T_0,x))$
$\Psi^n_{T_0}(t,x):= \Psi^n(t,\Phi^n(T_0,x))$ which satisfy (H11)-(H15), \eqref{eq:conv-diff} with $0$ and $T$ replaced by $T_0$ and $2T_0$. For every $x \in \overline{\Omega}$ we define
$\Theta^n_{T_0}(x):=\Phi^n(T_0,\Theta^n(\Psi(T_0,x)))$, $\Xi^n_{T_0}(x):=\Phi(T_0,\Xi^n(\Psi^n(T_0,x)))$  and we observe that they satisfy (H19)-(H23) and \eqref{hpboooo} with $0$ and $T$ replaced by $T_0$ and $2T_0$.

By \eqref{eq:convT0punt} we have that $(u^n(T_0),Du^n(T_0),\dot{u}^n(T_0)) \to (u(T_0),Du(T_0),\dot{u}(T_0))$ in $H\times \underline{H} \times H$ while \eqref{eq:conv-Vn}, \eqref{eq:conf-forz}, and \eqref{eq:convT0lp} give $G^n \to G$ in $H^1(0,T;\tilde{H})$. We are now in a position to apply Theorem \ref{thm:cont-Tpiccolo-NEW} on $[T_0,2T_0]$ to obtain
\begin{equation*}
    (   u^n(t),\, D u^n(t), \, \dot{ u}^n(t) ) \to (  u(t),\, D u(t), \, \dot{ u}(t) ) \quad \text{ in } H\times \underline{H} \times H,
 \end{equation*}
 for all $t\in [T_0,2T_0]$. Moreover there exists a constant $C>0$ such that
\begin{equation*}
    \| u^n(t) \| + \| D u^n(t)\| + \| \dot{ u}^n(t) \| \leq C
\end{equation*}
for every $n\in \mathbb{N}$ and $t \in [T_0,2T_0]$.
The conclusion can be obtained by itarating this process a finite number of times. 
\end{proof}

\vspace{1 cm}

\noindent \textsc{Acknowledgements.}
This paper is based on work supported by the National Research Project (PRIN  2017) 
``Variational Methods for Stationary and Evolution Problems with Singularities and 
 Interfaces", funded by the Italian Ministry of University and Research. 
The authors are members of the {\em Gruppo Nazionale per l'Analisi Ma\-te\-ma\-ti\-ca, la Probabilit\`a e le loro Applicazioni} (GNAMPA) of the {\em Istituto Nazionale di Alta Matematica} (INdAM).

\vspace{1cm}

{
}

\end{document}